\documentclass{amsart}
\usepackage{amsmath,amssymb}
\usepackage{bbm}
\usepackage{graphicx,tikz,mathtools}

\usepackage{color}
\usepackage{comment}
\definecolor{Caca}{RGB}{193,124,250}

\newtheorem{theorem}{Theorem}[section]
\newtheorem{proposition}[theorem]{Proposition}
\newtheorem{corollary}[theorem]{Corollary}

\theoremstyle{definition}

\theoremstyle{remark}
\newtheorem{remark}[theorem]{Remark}

\newcommand{\de}{\delta}
\newcommand{\ep}{\varepsilon}

\newcommand{\si}{\sigma}

\newcommand{\vp}{\varphi}

\newcommand{\De}{\Delta}

\newcommand{\La}{\cD}
\newcommand{\Si}{\Sigma}
\newcommand{\Om}{\Omega}

\def\RR{\mathbb{R}}

\newcommand{\cD}{{\mathcal D}}
\newcommand{\cE}{{\mathcal E}}

\newcommand{\cH}{{\mathcal H}}

\newcommand{\cM}{{\mathcal M}}
\newcommand{\cN}{{\#}}

\newcommand{\cV}{{\mathcal V}}

\newcommand{\pd}{\partial}
\newcommand\minus\backslash

\newcommand\lan\langle
\newcommand\ran\rangle

\renewcommand\leq\leqslant
\renewcommand\geq\geqslant
\newlength{\intwidth}

\numberwithin{equation}{section}

\newcommand\BSi\Si 

\newcommand\oSi{{\Sigma}}
\newcommand\ocM{{\cM}}

\newcommand\pdV{\pd_{\mathrm{V}}\Omega}
\newcommand\pdVc{\pd_{\mathrm{V}}\Omega^{\mathrm{c}}}
\newcommand\pdH{\pd_{\mathrm{L}}\Omega}
\newcommand\co{^{\mathrm{c}}}
\newcommand{\icM}{\dot\cM}

\begin{document}

\title[Nonexistence of Courant-type nodal domain bounds]{Nonexistence of Courant-type\\ nodal domain bounds for eigenfunctions\\ of the Dirichlet-to-Neumann operator}

 \author{Alberto Enciso}
 \address{Instituto de Ciencias Matem\'aticas, Consejo Superior de
   Investigaciones Cient\'\i ficas, 28049 Madrid, Spain, e-mail: {\sf aenciso@icmat.es}.}

\author{Angela Pistoia}
\address{Dipartimento di Scienze di Base e Applicate per l'Ingegneria, Universit\`a di Roma ``La Sapienza'',  00161 Roma, Italy, e-mail: {\sf angela.pistoia@uniroma1.it}.}

\author{Luigi Provenzano}
\address{Dipartimento di Scienze di Base e Applicate per l'Ingegneria, Universit\`a di Roma ``La Sapienza'',  00161 Roma, Italy, e-mail: {\sf luigi.provenzano@uniroma1.it}.}

%
%
\begin{abstract}
Given a compact manifold~$\cM$ with boundary of dimension $n\geq 3$ and any integers~$K$ and~$N$, we show that there exists a metric on~$\cM$ for which the first~$K$ nonconstant eigenfunctions of the Dirichlet-to-Neumann map on~$\pd\cM$ have at least~$N$ nodal components. This provides a negative answer to the question of whether the number of nodal domains of Dirichlet-to-Neumann eigenfunctions satisfies a Courant-type bound, which has been featured in  recent surveys  by Girouard and Polterovich~\cite[Open problem 9]{Pol17} and by Colbois, Girouard, Gordon and Sher~\cite[Open question 10.14]{GirouSurv}. 
\end{abstract}

\maketitle

\section{Introduction}

Let $\ocM$ be a compact $n$-dimensional manifold with boundary, endowed with a smooth Riemannian metric~$g$. The Dirichlet-to-Neumann (DtN) map of this Riemannian manifold is the linear operator $\cD: H^{\frac12}(\pd\cM)\to H^{-\frac12}(\pd\cM)$ 
 defined by $\La\vp:=\pd_\nu u$, where $u$ is the harmonic extension of~$\vp$, that is, the only solution to the boundary value problem
 \[
 \De u=0\quad\text{in }\icM\,,\qquad u=\vp\quad\text{on }\pd\cM\,.
 \]
 Here the Laplacian and the normal derivative are defined using the metric~$g$, $\icM$~denotes the interior of~$\ocM$, and $\nu$ is the outer unit normal to $\pd\cM$.
 
It is well known that~$\cD$ is a nonlocal pseudodifferential elliptic operator on~$\pd\cM$ of order~1, which defines a non-negative self-adjoint operator with dense domain in $L^2(\pd\cM)$. One can therefore take an orthonormal basis of $L^2(\pd\cM)$ consisting of eigenfunctions $\{\vp_k\}_{k=1}^\infty\subset  C^\infty(\pd\cM)$ of the DtN map, which satisfy the equation
\[
\cD\vp_k=\si_k\vp_k \quad\text{on }\pd\cM\,.
\]

The sequence of non-decreasing reals
\[
0=\si_0<\si_1\leq \si_2\leq \cdots\,,
\]
which tends to infinity as $k\to\infty$,
consists of the {\em Steklov eigenvalues}\/ of the manifold $(\cM,g)$. The harmonic extensions $u_k$ of the DtN eigenfunctions~$\vp_k$ satisfy
\begin{equation}\label{steklov}
\begin{cases}
\De u_k=0 & {\rm in\ }\icM\,,\\
\pd_\nu u_k=\si_k u_k & {\rm on\ }\pd\cM\,,
\end{cases}
\end{equation}
and are known as the {\em Steklov eigenfunctions}\/ of the Riemannian manifold $(\ocM,g)$. We refer to Problem \eqref{steklov} as the Steklov spectral problem. One can choose Steklov eigenfunctions $\{u_k\}_{k=1}^{\infty}\subset C^{\infty}(\cM)$ to form an orthonormal basis of the subspace of harmonic functions in $H^1(\cM)$. As is well known, $u_0$ is constant (and so is $\vp_0$).

In this note we are interested in the {\it geometry of  Steklov and DtN eigenfunctions}, specifically in the geometric properties of their nodal sets. The study of nodal sets of eigenfunctions is probably the oldest topic in spectral geometry, and can be traced back to Chladni's experiments with vibrating plates. The central result is Courant's nodal domain theorem, which asserts that the $k$-th Dirichlet eigenfunction of the Laplacian on a compact manifold with boundary~$\ocM$ has at most $k$ nodal domains (see \cite{Courant1923} or \cite[VI.6]{CoHil}). 

Since the proof of this landmark result only uses the min-max formulation of the  problem and unique continuation, it applies essentially verbatim~\cite{KS} to the case of Steklov eigenfunctions, yielding the nodal domain bound\footnote{Here and in what follows, we denote by $\cN(u_k)$ (respectively, $\cN(\vp_k)$) the number of nodal domains of the Steklov eigenfunction $u_k$ in~${\cM}$ (respectively, of the DtN eigenfunction~$\vp_k$ in~$\pd\cM$).}  $\cN(u_k)\leq k$. 

However, as the DtN map is nonlocal, the proof does not work in the case of DtN eigenfunctions. In dimension $n=2$, one can use a topological argument to show that the number of ``boundary nodal domains'' $\cN(\vp_k)$ can be bounded in terms of the number of ``bulk nodal domains'' $\cN(u_k)$ and of the topology of the manifold~$\cM$. In the particular case where the surface $\cM$ is simply connected, this translates~ \cite{AM_Stekloff} into the bound $\cN(\vp_k)\leq 2k$. We refer to \cite{sher} and \cite{KKP} for an account of results in the case of Riemannian surfaces. In fact, in dimension $2$,  the number of boundary nodal domains turns out to be related to the number of interior critical points of Steklov eigenfunctions~\cite{BPP}.

The situation is completely different in higher dimensions, as the number of nodal domains of a function in the bulk~$\cM$ does no longer control the number of boundary nodal components. A beautiful visual illustration of this principle was provided by Girouard and Polterovich in~\cite[Figure~6]{Pol17}, a minor variation of which we present here as Figure~\ref{F.GP}.

\begin{figure}\label{F.GP}
\centering 
		\includegraphics[width=0.4\textwidth]{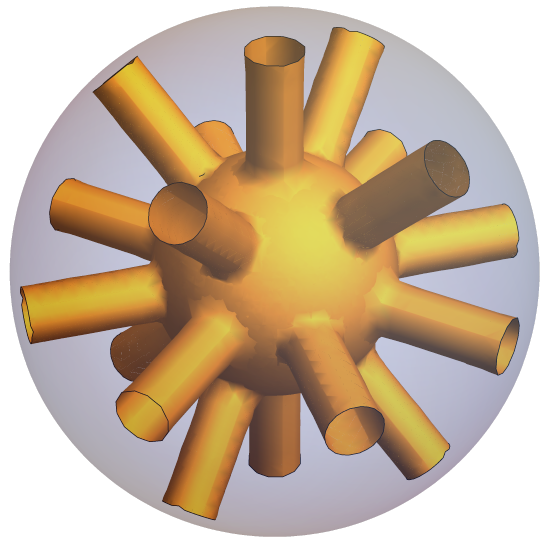}
	\caption{A surface inside a ball creating only two connected components in the interior and a large number of connected components on the boundary sphere, as illustrated in~\cite[Figure~6]{Pol17}.}
\end{figure}

In the last few years, the question of whether there is an analog of Courant's nodal domain bound for $\cN(\vp_k)$ has been posed in the authoritative surveys~\cite[Open problem 9]{Pol17} and~\cite[Open question 10.14]{GirouSurv}. Indeed, as argued in~\cite[Section~6]{Pol17}, there are indications that Courant's bound should hold for DtN eigenfunctions up to a universal constant depending on the dimension, i.e.,
\begin{equation}\label{conj}
\cN(\vp_k)\leq C_n(k+1)\,.
\end{equation}
This is what happens, for instance, in the case of Euclidean balls and cylinders.

Nevertheless, our objective in this note is to show that there are no universal nodal domain bounds for DtN eigenfunctions in dimension $n=3$ and higher. More precisely, one has the following:

\begin{theorem}\label{T.main}
Let $\ocM$ be a compact $n$-dimensional manifold with boundary, with $n\geq3$. Given any positive integer~$N$, there exists a smooth metric on~$\ocM$ such that the first nonzero eigenvalue $\si_1$ of the Dirichlet-to-Neumann map has multiplicity~$1$ and its corresponding eigenfunction~$\vp_1$ has exactly~$N$ nodal components. 

More generally, for any positive integers $N,K$ there exists a smooth metric on~$\ocM$ for which the first $K$ nonzero Steklov eigenvalues are simple and $\cN(\vp_k)\geq N$ for all $1\leq k\leq K$.
\end{theorem}

Roughly speaking, we obtain this result by showing that one can find a metric on~$\ocM$ for which the first~$K$ nonconstant Steklov eigenfunctions~$u_k$ have a level set that looks essentially like the hypersurface~$\Sigma$ depicted in Figure~\ref{F.GP}, with an arbitrary number~$N$ of boundary components. Thus Theorem~\ref{T.main} is an immediate consequence about a more general result on the nodal set of Steklov eigenfunctions:

\begin{theorem}\label{T.Steklov}
	Let $\ocM$ be a compact $n$-dimensional manifold with boundary, with $n\geq3$, and let $\oSi\subset\ocM$ be a compact connected separating hypersurface with boundary~$\pd\Sigma\subset\pd\cM$. We assume that $\Si$ intersects~$\pd\cM$ transversally. For any  positive integer $K$, there exists a metric on~$\ocM$ such that the first $K$ nonzero Steklov eigenvalues are simple and, for all $1\leq k\leq K$, $u_k$ has a nodal component isotopic to~$\Sigma$.
\end{theorem}

Note that $u_1$ cannot have any other nodal components because $\cN(u_1)=2$. In the statement of this theorem, we recall that a {\em nodal component}\/ of~$u_k$ is a connected component of the nodal set $u_k^{-1}(0)$ and that a compact hypersurface $\oSi\subset\ocM$ is {\em separating}\/ if $\ocM\backslash\oSi$ is disconnected. Also, two hypersurfaces with boundary $\oSi_0,\oSi_1$ are {\em isotopic}\/ if there is a smooth one-parameter family of diffeomorphisms $\Psi_t:\ocM\to\ocM$, with $t\in[0,1]$ and $\Psi_0=\text{identity}$, such that $\Psi_1(\oSi_0)=\oSi_1$. It is worth mentioning that the metrics can be chosen real analytic if the manifold~$\ocM$ is analytic.

In the proof of Theorem~\ref{T.Steklov}, which is fairly short, we elaborate on ideas developed in~\cite{JDG} to exploit the connection between Steklov eigenfunctions and the so-called {\it sloshing problem}, which is a classical eigenvalue problem with a long history in hydrodynamics. We refer to~\cite{sloshing} and to the references therein for historical information and for a discussion of physical applications.  A minor variation of the proof applies to the Steklov problem with a nonnegative potential, where one replaces the Laplacian in~\eqref{steklov} by $\Delta - q$, where $q:\cM\to[0,\infty)$ is continuous. Essentially, one only needs to ensure that the tubular neighborhood in the proof of Theorem~\ref{T.main} is narrow enough so that there the potential~$q$ in this neighborhood is almost independent of the ``horizontal coordinate'' (which we call~$t$).

Note that this result does not rule out the possibility that the Courant bound holds asymptotically, in the sense that the nodal count could be bounded as $O(k)$ as $k\to\infty$. This would certainly be the case if Steklov eigenfunctions behave asymptotically, in a suitable sense, as eigenfunctions of the Laplacian of the boundary, for which Courant's bound certainly holds. This asymptotic version version of the bound~\eqref{conj} has been explicitly conjectured in~\cite[Conjecture 1.4]{sher} (see also \cite[Section~6.1]{Pol17}), and a stronger version of the asymptotic result, of Pleijel-type, has been conjectured in~\cite [Conjecture~1.7]{sher}. In~\cite{PolPol}, it is shown that a large family of differential and pseudodifferential operators, including the DtN map, satisfy asymptotic Courant-type bound for ``deep'' nodal domains, that is, for nodal domains where the $L^\infty$-norm of an $L^2$-normalized eigenfunction is large enough.

To put the problem in perspective, let us mention that, although the Steklov problem was introduced well over a century ago \cite{steklov} to describe the stationary heat distribution in a body whose flux through the boundary is proportional to the temperature on the boundary, there has been a recent upsurge of activity on the spectral geometry of Steklov/DtN eigenfunctions, and the area is undergoing significant advancement. The motivation for this is twofold. On the one hand, as discussed at length in~\cite{KKKNPPS}, the DtN operator plays an essential role in a number of disparate field, such as medical and geophysical imaging~\cite{Uhl}, the analysis of water waves~\cite{Lannes}, or minimal surface theory~\cite{FS11,FS16}. On the other hand, Steklov eigenvalues exhibit a distinctly different (and very intriguing) behavior when contrasted with Laplace eigenvalues, so several fundamental issues are still insufficiently understood. However, in recent years there has been significant advancement in most aspects of the theory, including geometric bounds~\cite{Bellova, CGH, Legace,Miclo, KLP_weyl,RZ_Weyl, Zelditch,Zhu}, optimization problems~\cite{BDePhil, Petrides}, inverse spectral problems~\cite{Arias, Niky, Jammes}, and the behavior of high frequency eigenfunctions over the Planck scale and generic properties~\cite{Decio, Wang}. A wealth of information on these and other topics can be found in the surveys~\cite{GirouSurv,Pol17} and in the references therein.

The present note is organized as follows. The proof of Theorem~\ref{T.Steklov}  is presented in Section~\ref{S.proof}. In Sections~\ref{technical} and~\ref{S.technical2} we present the proofs of Propositions \ref{P.sloshing} and \ref{P.ep}, which are auxiliary results used in the proof of Theorem~\ref{T.Steklov}. We have also included two appendices which contain the proofs of some technical results.



\section{Proof of Theorem~\ref{T.Steklov}}
\label{S.proof}

In this section we shall prove Theorem~\ref{T.Steklov}. To streamline the presentation, the proofs of a couple of auxiliary results will be relegated to Sections~\ref{technical} and~\ref{S.technical2} below.

Before getting bogged down with technicalities, let us informally sketch the idea of the proof. We first take a  tubular neighborhood $\Omega$ of the separating hypersurface $\Sigma$ in $\cM$. We consider a Riemannian metric~$g$ on $\cM$ which coincides  with the pull-back of a product metric of the form $dt^2+\beta^2 g_{\Sigma}$ on the tubular neighborhood $\Omega$, where $g_{\Sigma}$ is any Riemannian metric on $\Sigma$. To ``penalize'' the eigenvalue problem, next we deform this metric by multiplying it by a small constant $\ep^2$ outside $\Omega$. The resulting metric~$g_\ep$ is then discontinuous on $\cM$. At this point we study the behavior of the Steklov eigenfunctions of $(\cM,g_\ep)$  as $\ep\to 0$, and prove that, inside $\Omega$, they converge to the eigenfunctions of a mixed Steklov--Neumann problem on $(\Omega,g)$, known as the {\it sloshing}\/ problem. If the parameter $\beta$ is small enough, it is possible to know precisely the geometry of the first $K$ nonzero sloshing eigenfunctions of $\Omega$. In particular, the first one has only one nodal set which is isotopic to $\Sigma$. Since we have shown that the Steklov eigenfunctions of $(\cM,g_\ep)$ are suitably small perturbations of sloshing eigenfunctions for small~$\ep$ on the tubular neighborhood~$\Om$, one can use Thom's isotopy theorem to obtain a similar statement about the eigenfunctions corresponding to the discontinuous metric~$g_\ep$. In the final step of the proof, we choose a family of smooth metrics~$g_{\ep,h}$ that approximate the discontinuous metric~$g_\ep$ in a suitable sense as $h\to0$, and show that the Steklov eigenfunctions of the smooth Riemannian manifold~$(\cM,g_{\ep,h})$ also have a nodal component diffeomorphic to~$\Si$ provided that $h$ is small enough.

Let us now present the details of the argument. For clarity, we will split the proof in fives steps.

\subsection*{Step 1: The discontinuous metric~$g_\ep$ and its Steklov eigenfunctions} Let us start by taking a thin closed tubular neighborhood ${\Omega_2}\subset\ocM$ of the hypersurface~$\oSi$ in~$\ocM$. Since $\Si$ intersects~$\pd\cM$ transversally, there exists a diffeomorphism $\Psi:\Omega_2\to [-2,2]\times\oSi$ such that $\Psi(\Sigma)=\{0\}\times\Sigma$. For any  $\tau\in(0,2)$, we will similarly denote
\[
\Om_\tau:=\Psi^{-1}([-\tau,\tau]\times\Si)\,.
\]
The interior of these closed sets will be denotes by $\dot\Omega_\tau$. When $\tau=1$, we will simply write $\Om:=\Om_1$.

If $\beta$ is a positive constant, if $t$ denotes the coordinate corresponding to the interval~$[-2,2]$ and if $g_\Si$ is {\it any} Riemannian metric on~$\oSi$, $dt^2+\beta^2 g_\Si$ defines a metric on~$\Om_\tau$. In what follows, let us fix a smooth metric~$g$ on~$\ocM$ which coincides on~$\Om_2$ with the pullback of this metric by the diffeomorphism~$\Psi$, that is,
\begin{equation}\label{E.g}
	g|_{\Om_2}=\Psi^*(dt^2+\beta^2 g_\Si )\,.
\end{equation}

Let us now define a discontinuous (but piecewise smooth) metric on~$\ocM$, depending on a small parameter $\ep>0$, as
\[
g_\ep:=\begin{cases}
	g& \text{on } \Om,\\
	\ep^2 g& \text{on } \Om\co ,
\end{cases}
\]
where $\Om\co :=\ocM\backslash \Om$ is the complement of $\Om$ in $\ocM$. Let us consider the Steklov problem associated with this metric, which is well defined because the variational formulation of this problem does not involve any derivatives of the metric.

More precisely, for any~$k\geq0$, the min-max characterization permits to write the Steklov eigenvalues~$\si_k^\ep$ associated with the discontinuous metric~$g_\ep$ on~$\ocM$ as
\begin{equation}\label{E.Rayleigh}
	\si_k^\ep=\inf_{E\in\cE_{k+1}(\cM)}\max_{0\ne u \in E} \frac{\int_{\cM}|\nabla_\ep u|_\ep^2\,dv_\ep}{\int_{\pd\cM} u^2\, d\si_\ep}\,.
\end{equation}
Here and in what follows, $\cE_{k+1}(\cM)$ denotes the set of $(k+1)$-dimensional linear subspaces of the Sobolev\footnote{We recall that the choice of a Sobolev $H^1$-norm on the manifold~$\cM$, which is usually done either using a partition of unity adapted to certain locally finite atlas, is highly non-intrinsic. However, since $\cM$ is compact, it is standard that the equivalence class of the $H^1$-norm is independent of these choices.} space~$H^1(\cM)$, and the subscripts mean that the norm of the gradient of~$u$ and the volume and area measures are computed using the metric $g_\ep$. If $u$ has zero trace to the boundary, the above Rayleigh quotient has to be interpreted as $+\infty$ (or, equivalently, one can just consider subspaces where the denominator does not vanish). As is well known, the above $\inf$ can be replaced by a $\min$, and the functions for which this $\min\max$ is attained are the Steklov eigenfunctions~$u_k^\ep$ corresponding to the eigenvalue~$\si_k^\ep$.

To study how the Rayleigh quotient \eqref{E.Rayleigh} depends on~$\ep$, it is convenient to write the boundary of~$\cM$ as the union of
\[
\pdV:=\pd\Omega\cap \pd\cM
\]
and $\pdVc:=\pd\Om\backslash\pdV$. The part of the boundary of~$\Om$ that is not in~$\pdV$ will be denoted by $\pdH$. In terms of the diffeomorphism $\Psi$, these boundary components can be written as
\[
\pdV=\Psi^{-1}([-1,1]\times\pd\Si)\,,\qquad \pdH=\Psi^{-1}(\{-1,1\}\times\oSi)\,.
\]

We can now make explicit the dependence of the Rayleigh quotient~\eqref{E.Rayleigh} on~$\ep$ as follows:
\begin{equation}\label{E.Rayep}
\frac{\int_{\cM}|\nabla_\ep u|_\ep^2\,dv_\ep}{\int_{\pd\cM} u^2\, d\si_\ep}=\frac{\int_{\Omega}|\nabla u|^2+\ep^{n-2}\int_{\Omega\co}|\nabla u|^2}{\int_{\pdV} u^2+ \ep^{n-1}\int_{\pdVc} u^2}\,.
\end{equation}
Here and throughout all the paper the norms and integrals defined by the metric~$g$, and the volume and area measures, are notationally omitted to make the expressions less cumbersome.

\subsection*{Step 2: The sloshing eigenvalue problem in $\Omega$.} The identity~\eqref{E.Rayep} suggests that, for very small~$\ep$, the Steklov eigenfunctions defined by the metric~$g_\ep$ should be connected with the formal limit problem
\begin{equation}\label{E.sloshing}
\mu_k(\beta)=\inf_{E\in\cE_{k+1}(\Om)}\max_{0\ne w\in E} \frac{\int_{\Omega}|\nabla w|^2}{\int_{\pdV} w^2}\,,
\end{equation}
where $\cE_{k+1}(\Om)$ is the set of $(k+1)$-dimensional subspaces of $H^1(\Om)$. It is known that the infimum is in fact attained, and this defines sloshing eigenfunctions~$w_k\in H^1(\Om)$. For future convenience,  we have highlighted the dependence on the constant $\beta$  (recall that $g_{\ep}=\Psi^*(dt^2+\beta^2g_{\Sigma})$ on $\Omega$). 
The eigenvalues $\mu_k(\beta)$ and the corresponding eigenfunction $w_k$ in \eqref{E.sloshing} are associated with the so-called {\it sloshing problem}\/:
\begin{equation}\label{E.sloshing.class}
\begin{cases}
\Delta w_k=0 & {\rm in\ }\dot \Omega\,,\\
\partial_{\nu}w_k=0 & {\rm on\ }\pdH\,,\\
\partial_{\nu}w_k=\mu_k(\beta)w_k & {\rm on\ }\pdV.
\end{cases}
\end{equation}

Before proving a precise convergence result, in the following proposition we analyze this formal limiting eigenvalue problem \eqref{E.sloshing.class}. The proof of this result is given in Section~\ref{technical} below. To state this result, we find it notationally convenient to label the eigenvalues by two nonnegative integers $\mu_j^\ell(\beta)$, with the corresponding eigenfunction $w_j^\ell$. Thus, as a set with multiplicities, $\{\mu_k(\beta)\}_{k=1}^\infty=\{\mu_j^\ell(\beta)\}_{j,\ell=0}^\infty$, but the connection between the positive integer~$k$ (which labels the eigenvalues so that they are nondecreasing) and the nonnegative integers $j,\ell$ (which provide a very neat characterization the sloshing spectrum with this product metric) is in general nontrivial.

\begin{proposition}\label{P.sloshing}
The eigenvalues of the sloshing problem~\eqref{E.sloshing} are
\begin{equation}\label{spec.slosh}
\left\{\mu^{\ell}_j(\beta)\right\}_{j,\ell=0}^{\infty}
\end{equation}
where $\mu_j^{\ell}(\beta)$ is the $j$-th Steklov eigenvalue on $\Sigma$ associated with the metric $\beta^2g_{\Sigma}$ and the constant potential $\frac{\pi^2\ell^2}{4}$. That is, $\{\mu_j^{\ell}(\beta)\}_{j=0}^\infty$ are the eigenvalues of the problem
\begin{equation}\label{E.sloshing.vj}
\begin{cases}
-\Delta v_j^\ell+\frac{\pi^2\ell^2}{4}v_j^\ell=0 & {\rm in\ }\dot \Si\,,\\
\partial_{\nu_\Si}v_j^\ell=\mu_j^{\ell}(\beta)v_j^\ell & {\rm on\ }\pd\Si\,,
\end{cases}
\end{equation}
where $\nu_\Si$ is the outer unit normal to $\Sigma$. The eigenfunction $w_j^\ell$ corresponding to $\mu_j^{\ell}(\beta)$ can then be written in terms of $v_j^\ell$ as
$$
w_j^\ell\circ\Psi^{-1}(t,y)=v_j^{\ell}(y)\cos\frac{\pi \ell (t+1)}{2}\,.
$$ 
Moreover, $\mu_0^0(\beta)=0$, $\mu_j^{\ell}(\beta)>0$  for all $(j,\ell)\ne(0,0)$, and $\inf_{y\in\Si} |v^\ell_0(y)|>0$ for all~$\ell$.

Furthermore, for any fixed integer $K$ there exists $\beta_0>0$ such that, if $\beta<\beta_0$, 
	\[
	 \mu_k(\beta)=\mu_0^{k-1}(\beta)\,,\qquad w_k=w_0^{k-1}\,.
	\]
for all $1\leq k\leq K$, and these eigenvalues have multiplicity~$1$.
\end{proposition}

Note that in \eqref{E.sloshing.vj} both the Laplacian and the normal derivative depend on $\beta$ (in fact, in Proposition \ref{P.sloshing} we are considering the metric $\beta^2d_{\Sigma}$ on $\Sigma$).

In what follows, we fix $K\in\mathbb N$ and $\beta<\beta_0$ as in Proposition \ref{P.sloshing}. For our purposes, the key property of the sloshing eigenfunctions defined by the metric~$g$ is then the following:

\begin{corollary}\label{C.Thom}
	There exists some $\de>0$ such that the nodal set $u^{-1}(0)$ of any smooth function $u:\Om_{\frac12}\to\RR$ satisfying
	\[
	\min_{1\leq k\leq K}\|u-w_k\|_{C^1(\Om_{\frac12})}<\delta
	\]
	has a connected component isotopic to~$\Sigma$.
\end{corollary}

\begin{proof}
By Proposition~\ref{P.sloshing}, $w_k\circ\Psi^{-1}(0,\cdot)=0$, which means that $\oSi$ is a nodal component of~$w_k$. Furthermore, $\nabla w_k$ does not vanish on~$\oSi$ because
\begin{equation*}
\min_{y\in\oSi}\pd_t(w_k\circ\Psi^{-1})(0,y)= \frac{\pi k}2\min_{\oSi}v_0^k>0\,.
\end{equation*}
Thus, the zero set of any function which is close enough in the~$C^1(\Om_{\frac12})$-norm to a sloshing eigenfunction $w_k$ (with $1\leq k\leq K$)  must have a connected component isotopic to~$\Si$ by Thom's isotopy theorem (see e.g.,~\cite[Section 20.2]{AR}). 
\end{proof}

\subsection*{Step 3: Convergence of the Steklov eigenfunctions of $(\cM,g_\ep)$ to  sloshing eigenfunctions.} The next ingredient of the proof is a result ensuring that for $\ep$ small the Steklov eigenfunctions are suitably close to those of the sloshing problem. The proof of this proposition is relegated to Section~\ref{S.technical2}.

\begin{proposition}\label{P.ep}
Suppose that $\mu_k(\beta)$ are simple for $1\leq k\leq K$. Then 
$$
\lim_{\ep\to 0}\sigma_k^{\ep}=\mu_k(\beta).
$$
Moreover, for all $1\leq k\leq K$, there exists $\ep_0>0$ such that for all $0<\ep<\ep_0$ the eigenvalue $\sigma_k^{\ep}$ is simple and there exists a family $\{u_k^{\ep}\}_{\ep}$ of corresponding eigenfunctions such that
$$
\lim_{\ep\to 0}
	\max_{1\leq k\leq K}\|u_k^\ep-w_k\|_{C^1(\Om_{\frac12})}=0.
$$
Here $w_k$ denote eigenfunctions of the sloshing problem \eqref{E.sloshing}, normalized so that $\|w_k\|_{L^2(\pdV)}=1$.
\end{proposition}

Hence, we have shown that, for any fixed integer $K$ there exists $\beta>0$ for which the first $K$ sloshing eigenvalues $\mu_k(\beta)$ are simple and that, for $\ep$ sufficiently small, the eigenfunctions $u_k^\ep$ of the Steklov problem for the singular metric $g_\ep$ are sufficiently close in $C^1(\Omega_{\frac{1}{2}})$ to $w_k$, the sloshing eigenfunctions associated with $\mu_k(\beta)$.

\subsection*{Step 4: Smoothing out the discontinuous metric~$g_\ep$.} Although the Steklov eigenfunctions $u_k^\ep$ do have the nodal components that we want, the metric~$g_\ep$ is discontinuous across the ``lateral boundary'' $\pdH$. 

To fix this issue, one only needs to smooth out this metric. To do so, let $\rho_{\ep,h}$ be a sequence of smooth functions such that $\rho_{\ep,h}^2$  converges pointwise to $\mathbbm 1_{\Omega}+\ep^2\mathbbm 1_{\Omega\co }$ as $h\to 0$, and such that $\rho_{\ep,h}^2=\mathbbm 1_{\Omega}+\ep^2\mathbbm 1_{\Omega\co }$ outside $\Omega_{1+\delta}\setminus\Omega_{1-\delta}$ ($0<\delta<1/2$), and that $\ep^2\leq\rho_{\ep,h}^2\leq 1$ in $\Omega\co $. Then we take
$$
g_{\ep,h}:=\rho_{\ep,h}^2g_{\ep},
$$
which is a sequence of smooth metrics which approximate the discontinuous metric $g_{\ep}$.

We will now use the following easy convergence result connecting the Steklov eigenfunctions  of the smoothed out metric $g_{\ep,h}$ to those of~$g_\ep$.

\begin{proposition}\label{P.h}
	For any given $\delta>0$ and for all small enough~$h>0$, there are Steklov eigenfunctions $u_k^{\ep,h}$ of the metric~$g_{\ep,h}$ such that
	\[
	\max_{1\leq k\leq K}\|u_k^{\ep,h}-u_k^\ep\|_{C^1(\Om_{\frac12})}<\frac\de2\,.
	\]
\end{proposition}

\begin{proof}
The result follows from standard results of approximation of eigenfunctions and eigenvalues of discontinuous metric, see e.g.,  \cite[I.8]{CdV} (see also \cite{CoEl}). In particular, we can apply \cite[I.8]{CdV} to the first $K$ nontrivial eigenfunctions $u_k^{\ep,h}$ when the  eigenvalues $\sigma_k^{\ep}$ are simple, which is guaranteed by Propositions \ref{P.sloshing} and \ref{P.ep} if $\ep,\beta$ are chosen sufficiently small. In this situation, from \cite[I.8]{CdV} we deduce that
$$
\lim_{h\to 0}\sigma_k^{\ep,h}=\sigma_k^{\ep}
$$
 and that there exists a sequence eigenfunctions $\{u_k^{\ep,h}\}$ associated with $\sigma_k^{\ep,h}$ such that
$$
\lim_{h\to 0}\|u_k^{\ep,h}-u_k^{\ep}\|_{H^1(\cM)}\to 0
$$
for all $k=1,...,K$. This in particular implies the claim of Proposition \ref{P.h} by standard elliptic regularity estimates away from $\pdH$ (see e.g.~\cite[Chapter~9]{GT}).
\end{proof}


\subsection*{Step 5: Conclusion of the proof.} Using the triangle inequality, we infer from Corollary~\ref{C.Thom} and Propositions~\ref{P.ep} and~\ref{P.h} that the smooth metric $g_{\ep,h}$ has the properties listed in Theorem~\ref{T.Steklov} for any small enough~$h>0$. This concludes the proof of Theorem \ref{T.Steklov}.

\section{Proof of Proposition~\ref{P.sloshing}}\label{technical}

The first part of the proposition is straightforward and follows from the characterization of the spectrum of the sloshing problem on a product manifold endowed with a product metric, see Appendix \ref{app_sloshing}. Namely, since the metric on $\Omega$ is the pull-back of a product metric, we can separate variables. We denote by $t$ the coordinate corresponding to $[-1,1]$ and by $y$ the coordinate on $\Sigma$. We obtain that the family of functions
$$
\left\{v_j^{\ell}(y)\cos\frac{\pi \ell (t+1)}{2}\circ\Psi\right\}_{j,\ell=0}^{\infty}
$$ 
are eigenfunctions of problem \eqref{E.sloshing} with corresponding eigenvalues $\mu_j^{\ell}(\beta)$, 
where $(v_j^{\ell},\mu_j^{\ell}(\beta))_{j=0}^ {\infty}$ are eigenpairs for the family of problems indexed by $\ell\in\mathbb N$
\begin{equation}\label{steklov.potential}
\begin{cases}
-\Delta v^\ell_j+\frac{\pi^2\ell^2}{4}v^\ell_j=0\,, & {\rm in\ }\dot \Sigma\,,\\
\partial_{\nu_\Si}v^\ell_j=\mu^{\ell}_j(\beta)v^\ell_j\,, & {\rm on\ }\partial\Sigma.
\end{cases}
\end{equation}
Recall that the metric on $\Sigma$ is $\beta^2g_{\Sigma}$, hence the Laplacian and the normal derivative depend on $\beta$. Here $\nu_\Si$ denotes the outer unit normal to $\partial\Sigma$, and $\mu^\ell_j(\beta)$ are of course ordered for fixed $\ell$ so that the sequence is nondecreasing. Then we re-order the eigenvalues of \eqref{E.sloshing.vj} and denote them by $\mu_k(\beta)$, and denote the corresponding eigenfunctions by $w_k$.  From Appendix \ref{app_sloshing} we also deduce that the functions $\{w_k\}_{k=0}^{\infty}$ are {\it all} the eigenfunctions of \eqref{E.sloshing}.

To prove the last statement of the proposition, we have to study the behavior of the eigenvalues and the eigenfunctions of \eqref{steklov.potential} as $\beta\to 0$. First, we note that when $\ell=0$ we have simply the Steklov problem on $\Sigma$ for the metric $\beta^2g_{\Sigma}$. The first eigenvalue is $0$, and the corresponding eigenfunction is constant. If $\ell\ne 0$, the first eigenvalue $\mu^{\ell}_0(\beta)$ is always strictly positive, it is simple, and a corresponding eigenfunction can be chosen positive in $\Sigma$ (this is a consequence of the maximum principle). Therefore with this choice we have $v^{\ell}_0(y)>0$ on $\Sigma$ for any $\ell\ne 0$. 

Writing the Rayleigh quotient for $\mu^{\ell}_j(\beta)$, we have 
\begin{equation}\label{R.eta}
\mu^{\ell}_j(\beta)=\inf_{E\subset\mathcal E_{j+1}(\Sigma)}\max_{0\ne v\in E}\frac{\int_{\Sigma}\beta^{-1}|\nabla v|_{g_{\Sigma}}^2+\frac{\pi^2\ell^2\beta}{4} v^2 dv_{g_{\Sigma}}}{\int_{\partial\Sigma}v^2 dv_{g_{\Sigma}}}\,,
\end{equation}
where the gradient is taken with respect to the {\it fixed} metric $g_{\Sigma}$ on $\Sigma$. Here $\mathcal E_{j+1}(\Sigma)$ denotes the set of $(j+1)$-dimensional subspaces of $H^1(\Sigma)$. We re-write \eqref{R.eta} as 
\begin{equation}\label{R.fix.metric}
\xi_j^{\ell}(\beta):=\beta\mu^{\ell}_j(\beta)=\inf_{E\subset\mathcal E_{j+1}(\Sigma)}\max_{0\ne v\in E}\frac{\int_{\Sigma}|\nabla v|_{g_{\Sigma}}^2+\frac{\pi^2\ell^2\beta^2}{4} v^2 dv_{g_{\Sigma}}}{\int_{\partial\Sigma}v^2 dv_{g_{\Sigma}}}\,.
\end{equation}
The right-hand side of \eqref{R.fix.metric} is the Rayleigh quotient of a Steklov problem with the potential $\frac{\pi^2\ell^2\beta^2}{4}$ on $\Sigma$ for a fixed metric $g_{\Sigma}$, whose eigenvalues are denoted by $\xi_j^{\ell}(\beta)$. 

As $\beta\to 0$, we conclude that this is a regular perturbation of the Steklov problem on $\Sigma$ for the fixed metric $g_{\Sigma}$. In particular we have that $\xi_j^{\ell}(\beta)\to \mu_j^0(1)$ as $\beta\to 0$ (recall that $\mu_j^0(1)$ are just the Steklov eigenvalues on $\Sigma$). Thus $\xi_j^{\ell}(\beta)\to 0$ if and only if $j=0$, which implies that 
$$
\mu_j^{\ell}(\beta)\to\infty\ \ \ \forall j\ne 0
$$
as $\beta\to 0$. 

On the other hand, using the constant function $v=1$ as test function for $\mu^{0}_j(\beta)$ in \eqref{R.eta}, we obtain
$$
\mu^{\ell}_0(\beta)\leq\frac{\pi^2\ell^2\beta|\Sigma|}{|\partial\Sigma|}.
$$
We have therefore proved that the sloshing eigenvalues \eqref{E.sloshing} have the following behavior as $\beta\to 0$:
$$
\mu_j^{\ell}(\beta)\to+\infty\,,\ \ \ {\rm if\ } j\geq 1
$$
and
$$
\mu_0^{\ell}(\beta)\to 0\,,\ \ \ {\rm if\ } j=0.
$$
This immediately implies the last part of the proposition, including the simplicity of the first $K$ eigenvalues (for any fixed $K\in\mathbb N$) if $\beta$ is chosen sufficiently small.

\section{Proof of Proposition~\ref{P.ep}}
\label{S.technical2}

The proof of this proposition uses ideas of \cite[Theorem 2.1]{JDG} (see also \cite[Proposition 2.2]{CoEl} and \cite[III]{CdV} for related problems). Let $\mu_k(\beta)$ be the sloshing eigenvalues (see Problem \ref{E.sloshing}), and let $w_k$ be the corresponding eigenfunctions normalized by $\int_{\pdV}w_k^2=1$. Therefore $\int_{\Omega}|\nabla w_k|^2=\mu_k(\beta)$. By hypothesis we have that the first $K$ sloshing eigenvalues are simple. To simplify the notation we shall write $\mu_k$, since $\beta$ is now fixed. 

Let $\psi_k\in H^1(\cM)$ be defined by
$$
\psi_k:=\begin{cases}
w_k\,, & {\rm in\ }\Omega\,,\\
\hat w_k\,, & {\rm in\ }\Omega\co ,
\end{cases}
$$
where the extension $\hat w_k$ solves
\begin{equation}\label{H.ext}
\begin{cases}
\Delta\hat w_k=0\,, & {\rm in\ }\Omega\co \,,\\
\hat w_k=w_k\,, & {\rm on\ }\pdH\,,\\
\partial_{\nu}\hat w_k=0\,, & {\rm on\ }\pdVc.
\end{cases}
\end{equation}
Note that, as the trace of $w_k$ on $\pdH$ belongs to $H^{1/2}(\pdH)$, problem \eqref{H.ext} admits a unique solution $\hat w_k\in H^1(\Omega_c)$. In fact, this is the unique solution of the variational problem of minimizing $\int_{\Omega\co }|\nabla \hat w|^2$ among all $\hat w\in H^1(\Omega_c)$ with trace equal to the trace of $w_k$ on $\pdH$ (see also Appendix \ref{estimate}). Here we are considering the metric $g$ on $\Omega\co $, and the Laplacian and the gradient are defined using this metric too.

Now, recall the following standard estimate:
\begin{equation}\label{Es.1}
\|\hat w_k\|^2_{H^1(\Omega\co )}\leq C\|w_k\|_{H^{1/2}(\pdH)}^2\leq C\|w_k\|_{H^1(\Omega)}^2\leq C(\mu_k+1)\,,
\end{equation}
where the constant may change from line to line as is customary (note that the constant may depend on $\beta$ which however we have fixed before perturbing the metric outside $\Omega$). For the reader's convenience, we have included a proof of \eqref{Es.1} in Appendix \ref{estimate}.
In particular, for the third inequality we have used the fact that $\|w\|_{H^1(\Omega)}^2$ and $\int_{\Omega}|\nabla w|^2+\int_{\pdV}w^2$ are two equivalent norms on $H^1(\Omega)$. We have  included a proof of this fact in Appendix \ref{estimate}. 

In order to prove the convergence of the eigenvalues $\sigma_k^\ep$ to $\mu_k$ we shall establish upper and lower bounds for $\sigma_k^\ep$.

We start with the upper bounds. Consider the subspace $V_k$ of $H^1(\cM)$ generated by the functions $\psi_1,...,\psi_k$ defined above. First, let us prove that this space has dimension $k$ for all $k\leq K$ and all small enough~$\ep$. To see this, let $\psi,\psi'\in V_k$. Using \eqref{Es.1} we get
\begin{align*}
\left|\int_{\partial\cM}\psi\psi'd\sigma_{\ep}-\int_{\pdV}\psi\psi'\right|&=\ep^{n-1}\left|\int_{\pdVc}\psi\psi'\right|\\
&\leq C\ep^{n-1}\|\psi\|_{H^1(\Omega\co )}\|\psi'\|_{H^1(\Omega\co )}\\
&\leq C\ep^{n-1}\|\psi\|_{H^1(\Omega\co )}\|\psi'\|_{H^1(\Omega\co )}\\
&\leq C\ep^{n-1}(\mu_k+1)\|\psi\|_{L^2(\pdV)}\|\psi'\|_{L^2(\pdV)}\\
&\leq C_k\ep^{n-1}\|\psi\|_{L^2(\pdV)}\|\psi'\|_{L^2(\pdV)}.
\end{align*}
Since $\int_{\pdV}\psi_i\psi_j=\int_{\pdV}w_iw_j=\delta_{ij}$, we deduce that for all $k\leq K$ and all small enough~$\ep$, $V_k$ has dimension $k$. Since $k\leq K$, we will henceforth write $C$ instead of~$C_k$, with the proviso that our estimates are not uniform in~$K$.

Taking $\psi=\psi'$, the above estimate yields 
\begin{equation}\label{est1}
\left|\int_{\partial\cM}\psi^2d\sigma_{\ep}-\|\psi\|^2_{L^2(\pdV)}\right|\leq C\ep^{n-1}\|\psi\|^2_{L^2(\pdV)}.
\end{equation}
A similar estimate, using \eqref{Es.1}, allows to prove that
\begin{equation}\label{est2}
\left|\int_{\cM}|\nabla \psi|_{\ep}^2dv_{\ep}-\int_{\Omega}|\nabla\psi|^2\right|\leq C_k\ep^{n-2}\|\psi\|_{L^2(\pdV)}^2.
\end{equation}
Using \eqref{est1} and \eqref{est2} in the min-max formulation \ref{E.Rayleigh} for the eigenvalues $\sigma_k^{\ep}$ and the fact that $V_{k+1}$ is a $k+1$-dimensional subspace of $H^1(\cM)$ we get
\begin{align}\label{U.bound}
\sigma_k^{\ep}&\leq\max_{0\ne\psi\in V_{k+1}}\frac{\int_{\cM}|\nabla \psi|_{\ep}^2dv_{\ep}}{\int_{\partial\cM}\psi^2d\sigma_{\ep}}\\
\notag &\leq\max_{0\ne\psi\in V_{k+1}}\frac{1}{1-C\ep^{n-1}}\frac{\int_{\Omega}|\nabla\psi|^2}{\int_{\pdV}\psi^2}+\frac{C\ep^{n-2}}{1-C\ep^{n-1}}\\
&\leq \max_{0\ne w\in W_{k+1}}(1+C\ep^{n-1})\frac{\int_{\Omega}|\nabla w|^2}{\int_{\pdV} w^2}+C\ep^{n-2}\notag\\
&=(1+C\ep^{n-1})\mu_k+C\ep^{n-2},\notag
\end{align}
where $W_{k+1}$ is the subspace of $H^1(\Omega)$ spanned by $w_1,...,w_{k+1}$.

We establish now a lower bound for $\sigma_k^{\ep}$. To do so, let us start with the claim that
\begin{equation}\label{E.claim}
H^1(\cM)=\cH_1\oplus\cH_2
\end{equation}
where
$$
\cH_1:=\{\psi\in H^1(\cM):\Delta\psi\equiv 0{\rm\ in\ }\Omega\co \}
$$
and
$$
\cH_2:=\{\psi\in H^1(\cM):\psi\in H^1_{0,{\rm L}}(\Omega\co )\,,\psi\equiv 0{\rm\ in\ }\Omega\},
$$
where $H^1_{0,{\rm L}}(\Omega\co )$ is the closure in  $H^1(\Omega\co )$ of the subspace of $C^{\infty}(\Omega_c)$ of functions vanishing in a neighborhood of $\pdH$. To see this, given $\psi\in H^1(\cM)$, we let $\psi_1$ be the only function such that $\psi_1=\psi$ in $\Omega$ and
\begin{equation}\label{psi1}
\begin{cases}
\Delta\psi_1=0\,,&{\rm in\ }\Omega\co \,,\\
\psi_1=\psi\,, & {\rm on\ }\pdH\,,\\
\partial_{\nu}\psi_1=0\,, & {\rm on\ }\pdVc.
\end{cases}
\end{equation}
If we now set $\psi_2:=\psi-\psi_1$, it is easy to check that $\psi_1\in \cH_1$ and $\psi_2\in\cH_2$. 
From the definition of $\cH_1$ and $\cH_2$ it also follows that
\begin{equation}\label{orthogonality}
\int_{\cM}\langle\nabla\psi,\nabla\psi'\rangle=0
\end{equation}
for all $\psi\in\cH_1$ and $\psi'\in \cH_2$. In \eqref{orthogonality} the gradient, scalar product and volume element are defined by the metric $g$. The claim~\eqref{E.claim} then follows.

Now, given any $\psi\in H^1(\cM)$, write
$$
\psi=\psi_1+\psi_2
$$
where $\psi_1\in\cH_1$ and $\psi_2\in \cH_2$. From the upper bound \eqref{U.bound} we get that for $\ep$ sufficiently small (depending on $K$)
$$
\sigma_k^{\ep}\leq\mu_k+1.
$$
Therefore we can write
$$
\sigma_k^{\ep}=\inf_{V\in\cV_{k+1}}\max_{0\ne\psi\in \cV_{k+1}}\frac{\int_{\cM}|\nabla_\ep \psi|_{\ep}^2dv_{\ep}}{\int_{\partial\cM}\psi^2d\sigma_{\ep}}
$$
where $\cV_{k+1}$ is a $k+1$ dimensional subspace of $H^1(\cM)$ such that 
$$
\frac{\int_{\cM}|\nabla_\ep \psi|_{\ep}^2dv_{\ep}}{\int_{\partial\cM}\psi^2d\sigma_{\ep}}\leq\mu_k+1
$$
for all $\psi\in\cV_{k+1}$. 

Note that if $\psi=\psi_1+\psi_2$ as above and
$$
\int_{\partial\cM}\psi_2^2d\sigma_{\ep}\geq c \int_{\partial\cM}\psi^2 d\sigma_{\ep}
$$
then
\begin{equation}\label{contr}
\frac{\int_{\cM}|\nabla_\ep \psi|_{\ep}^2dv_{\ep}}{\int_{\partial\cM}\psi^2d\sigma_{\ep}}\geq\frac{c}{\ep}\frac{\int_{\Omega\co }|\nabla\psi_2|^2}{\int_{\pdVc}\psi_2^2}\geq\frac{\lambda_1 c}{\ep}
\end{equation}
where $\lambda_1$ is the first Steklov--Dirichlet eigenvalue on $\Omega\co $, which is positive. Namely, $\lambda_1$ is the first eigenvalue of $\Delta u=0$ in $\Omega\co $, $u|_{\pdH}=0$, $\partial_{\nu}u=\lambda u$ on $\pdVc$. Therefore, if a function $\psi$ belongs to some subspace $\cV_{k+1}$,
\begin{equation}\label{Es.2}
\int_{\partial\cM}\psi_2^2d\sigma_{\ep}\leq C\ep \int_{\partial\cM}\psi^2 d\sigma_{\ep},
\end{equation}
since otherwise we would have a contradiction with \eqref{contr}.

The last ingredient in order to prove a lower bound for $\sigma_k^{\ep}$ is the following estimate:
\begin{equation}\label{Es.3}
\|\psi\|_{H^1(\Omega\co )}^2\leq C\|\psi_1\|^2_{L^2(\pdV)}.
\end{equation}
This is obtained exactly as \eqref{Es.1}, with $\psi_1$ defined by \eqref{psi1}.

Using \eqref{orthogonality}, \eqref{Es.2} and \eqref{Es.3}, we obtain
\begin{align*}
\frac{\int_{\cM}|\nabla_\ep \psi|_{\ep}^2dv_{\ep}}{\int_{\partial\cM}\psi^2d\sigma_{\ep}}&=\frac{\int_{\cM}|\nabla_\ep\psi_1|_{\ep}^2dv_{\ep}+\int_{\cM}|\nabla_\ep\psi_2|_{\ep}^2dv_{\ep}}{\int_{\partial\cM}\psi^2d\sigma_{\ep}}\\
&\geq\frac{\int_{\cM}|\nabla_\ep\psi_1|_{\ep}^2dv_{\ep}}{(1+C\ep^{1/2})\int_{\partial\cM}\psi_1^2d\sigma_{\ep}}\\
&=\frac{\int_{\Omega}|\nabla\psi_1|^2+\ep^{n-2}\int_{\Omega\co }|\nabla\psi_1|^2}{(1+C\ep^{1/2})\left(\int_{\pdV}\psi_1^2+\ep^{n-1}\int_{\pdVc}\psi_1^2\right)}\\
&\geq(1-C\ep^{1/2})\frac{\int_{\Omega}|\nabla\psi_1|^2}{\int_{\pdV}\psi_1^2}\,.
\end{align*}
From the min-max principle \eqref{E.Rayleigh}, we then infer 
\begin{equation}\label{L.bound}
\sigma_k^{\ep}\geq(1-C\ep^{1/2})\mu_k.
\end{equation}

Combining \eqref{U.bound} and \eqref{L.bound} we then get
\begin{equation}\label{e.convergence}
\lim_{\ep\to 0}\sigma_k^{\ep}=\mu_k.
\end{equation}
Now, since for $k\leq K$ all eigenvalues $\mu_k$ are simple by Proposition \ref{P.sloshing}, we have that for $\ep$ close to zero $\sigma_k^{\ep}$ are simple for $k\leq K$, and moreover that we can choose suitably normalized eigenfunctions $u_k^{\ep}$ such that
$$
\lim_{\ep\to 0}\|u_k^{\ep}-w_k\|_{H^1(\Omega)}=0,
$$
see e.g., \cite[III.1]{CdV}. Elliptic regularity implies that the convergence is also in $C^1(\Omega_{\frac12})$ (or in $C^r(\Om_\tau)$ for any fixed $r$ and any fixed $\tau<1$).  The proposition is then proven.

\appendix

\section{Sloshing eigenvalues and eigenfunctions on a product manifold}\label{app_sloshing}

The purpose of this Appendix is to describe the spectrum of the sloshing problem on a product of two Riemannian manifold endowed with the product metric.

Let $\Sigma$ be a compact, smooth, $(n-1)$-dimensional manifold with boundary $\partial\Sigma$, endowed with a Riemannian metric $g$. We consider the product manifold $\mathcal N=[-1,1]\times\Sigma$ endowed with the product metric $dt^2+g$. Hence $\mathcal N$ is a compact, smooth, $n$-dimensional Riemannian manifold with boundary $([-1,1]\times\partial\Sigma)\cup(\{-1,1\}\times\Sigma)$. We consider the sloshing problem on $\mathcal N$:
\begin{equation}\label{sloshingN}
\begin{cases}
\Delta f_k=0 & \text{ in   } \dot{\mathcal{N}}\,,\\
\pd_{\nu}f_k=0 & \text{ on }  \{-1,1\}\times\dot\Sigma\,,\\
\pd_{\nu}f_k=\mu_k f_k & \text{ on } (-1,1)\times\pd\Sigma.
\end{cases}
\end{equation}

As customary, we separate variables $(t,y)\in[-1,1]\times\Sigma$ and look for eigenfunctions $f$ of the form $f(t,y)=g(t)v(y)$. It is immediate to check that the functions $f_j^{\ell}(t,y):=\cos\frac{\pi\ell(t+1)}{2}v_j^{\ell}(y)$, $j,\ell\in\mathbb N$ are eigenfunctions of $\eqref{sloshingN}$, where, for any fixed $\ell\in\mathbb N$, $v_j^{\ell}$ are the eigenfunctions of
$$
\begin{cases}
-\Delta v_j^{\ell}+\frac{\pi^2\ell^2}{4}v_j^{\ell}=0 &  {\rm in\ }\dot\Sigma\,,\\
\partial_{\nu_{\Sigma}} v_j^{\ell}=\mu_j^{\ell}v_j^{\ell} & {\rm on\ }\partial\Sigma.
\end{cases}
$$
Here $\nu_{\Sigma}$ denotes the outer unit normal to $\partial\Sigma$. The eigenvalue corresponding to $f_j^{\ell}$ is then $\mu_j^{\ell}$. 

We show  now that $\{\cos\frac{\pi\ell(t+1)}{2}v_j^{\ell}(y)\}_{j,\ell=0}^{\infty}$ are {\it all} the eigenfunctions of \eqref{sloshingN}, and thus that $\{\mu_j^{\ell}\}_{j,\ell=0}^{\infty}$ exhaust {\it all} the spectrum.

It is well known that for any fixed $\ell\in\mathbb N$ the eigenfunctions can be chosen so that the boundary traces $\{v_j^{\ell}|_{\partial\Sigma}\}_{j=0}^{\infty}$ are an orthonormal basis of $L^2(\partial\Sigma)$. On the other hand, $\{\cos\frac{\pi\ell(t+1)}{2}\}_{\ell=0}^{\infty}$ is obviously an orthonormal basis of $L^2((-1,1))$.

We now claim that $\{\cos\frac{\pi\ell(t+1)}{2}v_j^{\ell}|_{\partial\Sigma}\}_{j=0}^{\infty}$ is an orthogonal basis of $L^2((-1,1)\times\partial\Sigma)$. To see this, suppose that $f\in L^2(\mathcal N)$ is such that
$$
\int_{-1}^1\int_{\pd\Sigma} f(t,y)\cos\frac{\pi\ell(t+1)}{2}v_j^{\ell}(y)dtd\sigma_g(y)=0
$$
for all $\ell,j\in\mathbb N$. Here $d\sigma_g$ denotes the $(n-1)$-dimensional volume element of $\partial\Sigma$. By Fubini's Theorem
$$
\int_{\pd\Sigma}v_j^{\ell}(y)\left(\int_{-1}^1 f(t,y)\cos\frac{\pi\ell(t+1)}{2}dt\right)d\sigma_g(y)=0\,,\ \ \ \forall\ell,j\in\mathbb N.
$$
Since for all fixed $\ell\in\mathbb N$,  $\{v_j^{\ell}|_{\pd\Sigma}\}_{j=0}^{\infty}$ is a complete system and since it is easily seen by H\"older's inequality that $\int_{-1}^1 f(t,y)\cos\frac{\pi\ell(t+1)}{2}dt\in L^2(\pd\Sigma)$, we deduce that for any $\ell\in\mathbb N$
\begin{equation}\label{B0}
\int_{-1}^1 f(t,y)\cos\frac{\pi\ell(t+1)}{2}dt=0
\end{equation}
for almost every $y\in\pd\Sigma$. Since $\mathbb N$ is countable, \eqref{B0} holds for all $\ell\in\mathbb N$ almost everywhere in $\pd\Sigma$. Thus, since $\{\cos\frac{\pi\ell(t+1)}{2}\}_{\ell=0}^{\infty}$ is a complete system in $L^2((-1,1))$ we have that $f(t,y)=0$ almost everywhere in $(-1,1)\times\pd\Sigma$. This proves the claim.

Now, since we know that the traces of the eigenfunctions of \eqref{sloshingN} form a complete system in $(-1,1)\times\pd\Sigma$, we conclude that there are no other eigenfunctions than  $\{\cos\frac{\pi\ell(t+1)}{2}v_j^{\ell}(y)\}_{j,\ell=0}^{\infty}$.

\section{Estimates for solutions to a mixed problem}\label{estimate}

The aim of this section is to prove the chain of inequalities \eqref{Es.1}.

Let $(\cM,g)$ be a smooth $n$-dimensional Riemannian manifold (with or without boundary) and let $\Omega\subset \cM$ be a piecewise smooth, Lipschitz domain. Let $\Gamma\subset\partial \cM$ be an open, smooth, subset of $\partial\Omega$.

We start by proving that the standard norm $\|u\|_{H^1(\cM)}^2:=\int_{\cM}|\nabla u|^2+u^2$ of $H^1(\cM)$ is equivalent to the norm 
\begin{equation}\label{equiv_norm}
\|u\|_{H^1(\cM),\Gamma}^2:=\int_{\cM}|\nabla u|^2+\int_{\Gamma}u^2.
\end{equation}
We claim that there exists constants $c,C>0$ not depending on $u$ such that
$$
c\|u\|_{H^1(\cM)}^2\leq\|u\|_{H^1(\cM),\Gamma}^2\leq C\|u\|_{H^1(\cM)}^2
$$
We start from the first inequality. Assume by contradiction that it is not true. Then there exists a sequence $\{u_k\}\subset H^1(\cM)$ such that
$$
\|u_k\|_{H^1(\cM)}^2>k\|u_k\|_{H^1(\cM),\Gamma}^2.
$$
Without loss of generality we may assume $\|u_k\|_{H^1(\cM)}=1$ for all $k$. Hence we have that
$$
\|u_k\|_{H^1(\cM),\Gamma}^2<\frac{1}{k}.
$$
This implies that $\|\nabla u_k\|_{L^2(\cM)}\to 0$ and $\|u_k\|_{L^2(\Gamma)}\to 0$. Since $\{u_k\}$ is bounded in $H^1(\cM)$, up to subsequence, $u_k\rightharpoonup u$ in $H^1(\cM)$ and $u_k\to u$ in $L^2(\cM)$ and also in $L^2(\Gamma)$. Therefore $u\equiv 0$ on $\Gamma$ and $u\equiv{\rm const}$ in $\cM$. Hence $u\equiv 0$ in $\cM$. On the other hand, we have $\|u\|_{L^2(\cM)}=1$, which gives a contradiction. The second inequality is proved in the same way.

 Let now $\Gamma_1,\Gamma_2$ be two non-empty smooth open subsets of $\partial\Omega$, with $\overline{\Gamma}_1\cup\overline{\Gamma}_2=\partial\Omega$, $\Gamma_1\cap\Gamma_2=\emptyset$, $\overline{\Gamma}_1\cap\overline{\Gamma}_2$ is a smooth $(n-2)$-dimensional submanifold of $\partial\Omega$. Let $g\in H^{1/2}(\Gamma_1)$.
We consider problem
\begin{equation}\label{mixed}
\begin{cases}
\Delta u=0 & {\rm in\ }\Omega\,,\\
u=g & {\rm on\ }\Gamma_1\,,\\
\partial_{\nu}u=0 & {\rm on\ }\Gamma_2,
\end{cases}
\end{equation}

Problem \eqref{mixed} is the classical formulation of the following problem: find $u\in H^1(\Omega)$ such that
$$
\int_{\Omega}|\nabla u|^2=\min_{v\in V_g}\int_{\Omega}|\nabla v|^2,
$$
where
$$
V_g:=\{v\in H^1(\Omega):{\rm Tr}\,v=g{\rm\ on\ }\Gamma_1\}.
$$
Here ${\rm Tr}$ is the trace operator, and since $\Gamma_1$ is smooth and $g\in H^{1/2}(\Gamma_1)$, the set $V_g$ is not empty. Hence the minimization problem has a unique solution. The unique minimizer $u$ solves
$$
\int_{\Omega}\langle\nabla u,\nabla\phi\rangle=0
$$
for all $\phi\in C^{\infty}(\Omega)$ which vanish on a neighborhood of $\Gamma_1$, and ${\rm Tr}\,u=g$ on $\Gamma_1$.

We know that the trace operator ${\rm Tr}:H^1(\Omega)\to\Gamma_1$ is bounded and is surjective onto $H^{1/2}(\Gamma_1)$. On the other hand we have shown that it is bijective from the subspace 
$$
\left\{u\in H^1(\Omega):\int_{\Omega}\langle\nabla u,\nabla\phi\rangle=0 \,,\ \forall\phi\in C^{\infty}(\Omega)\,,\ \phi\equiv 0{\rm\ on\ a\ neighborhood\ of\ }\Gamma_1\right\}
$$
of $H^1(\Omega)$ to $H^{1/2}(\Gamma_1)$. Therefore its inverse is continuous on $H^{1/2}(\Gamma_1)$. We conclude that there exists $C>0$ independent on $g\in H^{1/2}(\Gamma_1)$ such that
\begin{equation}\label{inverse_bound}
\|u\|_{H^1(\Omega)}\leq C\|g\|_{H^{1/2}(\Gamma_1)}.
\end{equation}

Inequality \eqref{Es.1} then follows by first applying \eqref{inverse_bound} with $\Omega$ replaced by $\Omega\co$ and with $\Gamma_1=\pdH$; then using the continuity of the trace operator from $H^1(\Omega)$ to $H^{1/2}(\pdH)$, and then using the equivalence of the $H^1(\Omega)$ norm with the norm~\eqref{equiv_norm} with $\Gamma:=\pdV$.

\section*{Acknowledgements}

The authors are grateful to Alexandre Girouard and Iosif Polterovich, who have read a preliminary version of the manuscript, providing very valuable comments.

This work has received funding from the European Research Council (ERC) under the European Union's Horizon 2020 research and innovation programme through the grant agreement~862342 (A.E.). The first author is also partially supported by the grants CEX2019-000904-S, RED2022-134301-T and PID2022-136795NB-I00 funded by MCIN/AEI. The second author acknowledges support of 
INDAM-GNAMPA project ``Problemi di doppia curvatura su variet\`a  a bordo e legami con le EDP di tipo ellittico'' and of the project 
``Pattern formation in nonlinear phenomena'' funded by the MUR Progetti di Ricerca di Rilevante Interesse Nazionale (PRIN) Bando 2022 grant 
20227HX33Z. The third author acknowledges support of the INDAM-GNSAGA project ``Analisi Geometrica: Equazioni alle Derivate Parziali e Teoria delle Sottovariet\`a'' and of the the project ``Perturbation
problems and asymptotics for elliptic differential equations: variational and potential theoretic methods'' funded by the MUR Progetti di Ricerca di Rilevante Interesse Nazionale (PRIN) Bando 2022 grant 2022SENJZ3.

\bibliography{bibliography}{}
\bibliographystyle{abbrv}












\end{document}